\documentclass[12pt,a4paper]{amsart}
\usepackage{graphics}
\usepackage{epsfig}
\usepackage{graphicx}
\usepackage{amssymb}
\usepackage[all]{xy}

\advance\hoffset-20mm \advance\textwidth40mm

\newcommand{\SL}{\mathrm{SL}}
\newcommand{\Sp}{\mathrm{Sp}}
\newcommand{\SO}{\mathrm{SO}}
\newcommand{\PSU}{\mathrm{PSU}}

\newcommand{\kk}{\mathbb{K}}
\newcommand{\ZZ}{\mathbb{Z}}
\newcommand{\QQ}{\mathbb{Q}}
\newcommand{\PP}{\mathbb{P}}
\newcommand{\XX}{\mathcal{X}}
\newcommand{\YY}{\mathcal{Y}}
\newcommand{\NN}{\mathcal{N}}
\newcommand{\MM}{\mathcal{M}}
\newcommand{\Cl}{\mathrm{Cl}}
\newcommand{\quot}{/\!\!/}
\theoremstyle{plain}
\newtheorem{theorem}{Theorem}
\newtheorem{proposition}{Proposition}
\newtheorem{lemma}{Lemma}

\theoremstyle{definition}
\newtheorem{definition}{Definition}
\newtheorem{example}{Example}
\theoremstyle{remark}
\newtheorem{remark}{Remark}

\title{Homogeneous toric varieties}

\author[I.V. Arzhantsev and S.A. Gaifullin]{Ivan V. Arzhantsev
\and Sergey A. Gaifullin}

\thanks{Supported by RFBR grants 09-01-00648-a, 08-01-91855KO-a, and the Deligne
2004 Balzan prize in Mathematics.}


\address{Ivan V. Arzhantsev \\
Department of Higher Algebra, Faculty of Mechanics and Mathematics, Moscow State
University, Leninskie Gory 1, Moscow, 119991, Russia}
\email{arjantse@mccme.ru}

\address{Sergey A. Gaifullin \\
Department of Higher Algebra, Faculty of Mechanics and Mathematics, Moscow State
University, Leninskie Gory 1, Moscow, 119991, Russia}
\email{sgayf@yandex.ru}

\begin{document}

\begin{abstract}
A description of transitive actions of a semisimple
algebraic group $G$ on toric varieties is obtained.
Every toric variety admitting such an action lies between a product
of punctured affine spaces and a product of projective spaces.
The result is based on the Cox realization of a toric variety as a quotient
space of an open subset of a vector space $V$ by a
quasitorus action and on investigation of the $G$-module structure of $V$.
\end{abstract}

\maketitle

\section{Introduction}
We study toric varieties $X$ equipped with a
transitive action of a connected semisimple algebraic group $G$.
In this case $X$ is called  a {\it homogeneous toric
variety}. The ground field $\mathbb{K}$ is algebraically closed and
of characteristic zero.

Consider a quasiaffine variety
$$
\XX = \XX(n_1,\dots,n_m) :=
(\mathbb{K}^{n_1}\setminus \{0 \} )\times\dots\times(\mathbb{K}^{n_m}\setminus\{0\})
$$
with $n_i\geq 2$. The group $G=G_1\times\ldots\times G_m$,
where every component $G_i$ is either $\SL(n_i)$ or $\Sp(n_i)$, and $n_i$ is even in the second case,
acts on $\XX$ transitively and effectively. Let
$\mathbb{S} = (\kk^{\times})^m$ be an
algebraic torus acting on $\XX$ by component-wise scalar multiplication,
and
$$
p \, \colon \, \XX \, \to \ \YY \ := \ \PP^{n_1-1}\times\dots\times\PP^{n_m-1}
$$
be the quotient morphism. Fix a closed subgroup $S\subseteq\mathbb{S}$.
The action of the group $S$ on $\XX$ admits a geometric quotient
$p_X \colon \XX \to X := \XX / S$. The variety $X$ is toric, it carries
the induced action of the quotient group $\mathbb{S}/S$, and there
is a quotient morphism $p^X \colon X \to \YY$ for this action closing the commutative diagram

$$
\xymatrix{
\XX \ar[dr]_p \ar[rr]^{p_X} & & X \ar[dl]^{p^X} \\
 & \YY & }
$$

The induced action of the group $G$ on $X$ is transitive and locally effective.
We say that the $G$-variety $X$ is obtained from $\XX$ by {\it central factorization}.
The following theorem gives a classification of transitive actions of semisimple groups on
toric varieties up to a twist by a diagram automorphism of the acting group.

\begin{theorem}
\label{main}
Let $X$ be a toric variety with a transitive locally
effective action of a connected simply connected semisimple algebraic group $G$.
Then $G=G_1\times\ldots\times G_m$,
where every simple component $G_i$ is either $\SL(n_i)$ or $\Sp(n_i)$,
and the variety $X$ is obtained from
$\XX = \XX(n_1,\ldots,n_m)$ by central factorization.
Conversely, any variety
obtained from $\XX$ by central factorization is a homogeneous toric variety.
\end{theorem}

Theorem~\ref{main} also describes homogeneous spaces of a semisimple group
that have a toric structure. It is natural to apply the Cox realization
of a variety in order to search for toric varieties in a given class
of varieties. This idea is already used in \cite{Ga}, where
toric affine $\text{SL}(2)$-embeddings are characterized.

In Section~\ref{two} we recall basic facts on the Cox realization
and its generalization. Criterions of existence of an open $G$-orbit on $X$
in terms of $G$- and $(G\times S)$-actions on the total coordinate
space $Z$ are also given there.
In Section~\ref{three} we prove Theorem~\ref{main}.
The next section is devoted to special classes of toric homogeneous
varieties and to a characterization of their fans.
In the last section we consider transitive actions of
reductive groups on toric varieties.

Our results are closely connected with the results
of E.B.~Vinberg \cite{Vin}, where algebraic transformation groups of
maximal rank were classified. Recall that an {\it algebraic
transformation group of maximal rank} is an effective generically
transitive (i.e., with an open orbit) action of an algebraic group $\mathcal{G}$ on an algebraic
variety $X$ such that $\dim X=\mathrm{rk}\, \mathcal{G}$, where $\mathrm{rk}\, \mathcal{G}$
is the rank of a maximal torus $T$ of the group $\mathcal{G}$. In this
situation the induced action of the torus $T$ on $X$ is effective and generically
transitive, see~\cite{Dem}. If the group $\mathcal{G}$ is semisimple,
then an open $\mathcal{G}$-orbit on $X$ is a homogeneous toric variety.
It turns out that in this case $X$ is a product of projective spaces
and $\mathcal{G}$ acts on $X$ transitively. Theorem~\ref{main}
implies that every homogeneous toric variety determines a
reductive transformation group of maximal rank; here $\mathcal{G}$ is
the quotient group $(\mathrm{GL}(n_1)\times\ldots\times\mathrm{GL}(n_m))/S.$

Finally, let us mention a related result from toric topology. A
{\it torus manifold} is a smooth real even-dimensional manifold
$M^{2n}$ with an effective action of a compact torus $(S^1)^n$
such that the set of $(S^1)^n$-fixed points is nonempty.
In~\cite{Kur}, homogeneous torus manifolds are
studied. The latter are torus manifolds $M^{2n}$ with a
transitive action of a compact Lie group $K$ such that the
induced action of a maximal torus of $K$ coincides
with the given $(S^1)^n$-action. It is proved that every
homogeneous torus manifold may be realized as
$$
M=\mathbb{C}\mathbb{P}^{n_1}\times\ldots\times\mathbb{C}\mathbb{P}^{n_k}
\times(S^{2m_1}\times\ldots\times S^{2m_l})/F,
$$
where $S^{2m}$ is a sphere of dimension $2m$, $F$ is a
subgroup of $\mathbb{Z}_2\times\ldots\times\mathbb{Z}_2$ ($l$
copies), and each copy of $\mathbb{Z}_2$ acts on the
corresponding sphere by central symmetry. A compact Lie group
$$
K=\PSU(n_1+1)\times\ldots\times \PSU(n_k+1)\times \SO(2m_1+1)\times\ldots\times \SO(2m_l+1)
$$
acts on $M$ transitively. Moreover, the manifold $M$ is orientable if and only if
$F \subset \SO(2m_1+2m_2+\ldots+2m_l+l)$.

The authors are grateful to E.B.Vinberg and to the referee for useful comments and suggestions.

\section{The Cox construction}
\label{two}

A {\it toric variety} is a normal
algebraic variety with an effective generically transitive action of an
algebraic torus $T$. A toric variety $X$ is {\it non-degenerate} if
any invertible regular function on $X$ is constant.

Let $\Cl(X)$ be the divisor class group of the variety $X$.
It is well-known that the group $\Cl(X)$ of a toric variety $X$
is finitely generated,
see~\cite[Section~3.4]{Ful}. Recall that a {\it quasitorus} is an
affine algebraic group $S$ isomorphic to a direct product of an algebraic
torus $S^0$ and a finite abelian group $\Gamma$. Every closed subgroup of a torus is
a quasitorus. The group of characters of
a quasitorus $S$ is a finitely generated  abelian group.
The {\it Neron-Severi quasitorus} of a toric variety $X$ is a quasitorus
$S$ whose group of characters is identified with $\Cl(X)$.

We come to a canonical quotient realization
of a non-degenerate toric variety $X$ obtained in~\cite{cox}.
Let $d$ be the number of
prime $T$-invariant Weil divisors on $X$. Consider the vector space
$\kk^d$ and the torus $\mathbb{T} = (\kk^{\times})^d$ of all invertible
diagonal matrices acting on $\kk^d$. Then there are a closed embedding
of the Neron-Severi quasitorus $S$ into $\mathbb{T}$ and an open
subset $U \subseteq \kk^d$ such that
\begin{itemize}
\item{}
the complement $\kk^d \setminus U$ is a union of some coordinate
subspaces of dimension $\le d-2$;
\item{}
there exist a categorical quotient $p_X \colon U \to U \quot S$
and an isomorphism $\varphi \colon X \to U \quot S$;
\item{}
via isomorphism $\varphi$, the $T$-action on $X$ corresponds to the
action of the quotient group $\mathbb{T}/S$ on  $U \quot S$.
\end{itemize}

Later this realization was generalized to a
wider class of normal algebraic varieties, see \cite{Hu}, \cite{BH2},
\cite{Ha}. One of the conditions that determines this class is
finite generation of the divisor class group $\Cl(X)$.
This allows to define the Neron-Severi quasitorus $S$ of the variety $X$.
The space $\kk^d$ is replaced by an affine factorial (or,
more generally, factorially graded, see~\cite{Ar}) $S$-variety $Z$.
It is called  the {\em total coordinate space\/} of the variety $X$.
Further, $X$ appears as the quotient space of the categorical quotient
$p_X \colon U \to U \quot S$, where $U$ is an open $S$-invariant
subset of $Z$ such that the complement $Z \setminus U$ is of codimension
at least two in $Z$. The morphism $p_X \colon U \to X \cong U \quot S$
is called the {\em universal torsor\/} over $X$.

Let a connected affine algebraic group $G$ act on
a normal variety $X$. Passing to a finite covering we may
assume that $\Cl(G)=0$ \cite[Proposition~4.6]{KKLV}. Then the action of $G$
on $X$ can be lifted to an action of $G$ on the total
coordinate space $Z$ that commutes
with the $S$-action, see~\cite[Section~4]{ArH}.
It turns out that the set $U$ is $(G\times S)$-invariant and
%the universal torsor
$p_X \colon U \to X$ is a $G$-equivariant morphism.

\begin{lemma}
\label{lemac}
The following conditions are equivalent.
\begin{enumerate}
\item[(i)]
The action of the group $G$ on $X$ is generically transitive.
\item[(ii)]
The action of the group $G\times S$ on $Z$ is generically transitive.
\end{enumerate}
\end{lemma}

\begin{proof}
Let $X_0\subseteq X$ be an open $G$-orbit. Each point $x\in X_0$
is smooth on $X$, and thus the fiber $p_X^{-1}(x)$ is
isomorphic to the quasitorus $S$~\cite[Proposition~2.2,~(iii)]{Ha}.
It shows that the group $G\times S$ acts on $p_X^{-1}(X_0)$ transitively.

Conversely, if $Z_0\subseteq Z$ is an open $(G\times S)$-orbit,
then $Z_0\subseteq U$ and the action of $G$ on the quotient space
$U \quot S$ is generically transitive.
\end{proof}

Assume that the group $G$ has trivial group of characters.
Then the lifting of the action of the group $G$ to $Z$ is unique,
compare \cite[Remark~4.1]{ArH} and \cite[Proposition~1.8]{Ha1}.
Let $H$ be a closed subgroup of $G$.
Every invertible regular function on the homogeneous space
$G/H$ is constant, see~\cite[Proposition~1.2]{Kn}.

\begin{proposition}
\label{a}
The following conditions are equivalent.
\begin{enumerate}
\item[(i)]
The action of the group $G$ on $X$ is generically transitive and
the complement of an open $G$-orbit has codimension at least two in $X$.
\item[(ii)]
The action of the group $G$ on the total coordinate space $Z$
is generically transitive.
\item[(iii)]
The action of the group $G$ on the total coordinate space $Z$
is generically transitive and the complement of an open $G$-orbit
has codimension at least two in $Z$.
\end{enumerate}
\end{proposition}

\begin{proof}
We check "(i) $\Rightarrow$ (iii)".
Let $X_0\subseteq X$ be an open $G$-orbit.
The condition $\mathrm{codim}_X(X\setminus X_0)\geq 2$ implies that
$p_X \colon p_X^{-1}(X_0) \to X_0$ is the universal torsor
over $X_0$ and that the complement to
$p_X^{-1}(X_0)$ in $Z$ does not contain divisors, see~\cite[Section~2]{ARH}.
By~\cite[Lemma~3.14]{ARH} (see also~\cite[Theorem~4.1]{Ar}),
the universal torsor over a homogeneous space $G/H$ is the projection
$G/H_1\to G/H$, where $H_1$ is the intersection of kernels
of all characters of the subgroup $H$. This shows that
the group $G$ acts on $p_X^{-1}(X_0)$ transitively.

In order to obtain "(iii) $\Rightarrow$ (i)" note that $p_X(Z_0)$,
where $Z_0$ is the open $G$-orbit in $Z$, is an open $G$-orbit in $X$
whose complement does not contain divisors.
The implication "(iii) $\Rightarrow$ (ii)" is obvious.

To verify  "(ii) $\Rightarrow$ (iii)"
let $Z_0\subseteq Z$ be an open $G$-orbit. Since the
subset $Z_0$ is $S$-invariant, for every prime divisor $D\subset Z$
in the complement to $Z_0$ the set $S\cdot D$ is an $S$-invariant Weil divisor.
Each $S$-invariant Weil divisor on $Z$ is a principal divisor
$\mathrm{div}(f)$ of a regular function $f \in \kk[Z]$, see
\cite[Proposition~2.2,~(iv)]{Ha}. Then the non-constant function
$f$ is invertible on $Z_0$, a contradiction.
\end{proof}

The same arguments lead to the following result.

\begin{proposition}
\label{propprop}
The action of the group $G$ on $X$ is transitive if and
only if the open subset $U\subseteq Z$ is a $G$-orbit.
\end{proposition}

\section{Classification of homogeneous toric varieties}
\label{three}

In this section we prove Theorem~\ref{main}.
Since the variety $X$ is toric, its total coordinate
space $Z$ is an affine space.

\begin{lemma}
Let a semisimple group $G$ act on a toric variety $X$ with
an open orbit. Then $X$ is non-degenerate and the action of the
group $G \times S$ on the affine space $Z$ is equivalent to
a linear one.
\end{lemma}

\begin{proof}
Since any invertible function of the open $G$-orbit is constant,
the variety $X$ is non-degenerate.
By Lemma~\ref{lemac}, the action of the group $G \times S$ on
the space $Z$ is generically transitive, and the second statement follows
from~\cite[Proposition~5.1]{Kr}.
\end{proof}

Later on we assume that $G = G_1\times \ldots \times G_m$
acts on $X$ transitively.
Denote by $V$ the total coordinate space $Z$ of the variety $X$
regarded as the $(G\times S)$-module.
We proceed with a description of the $G$-module structure on $V$.

\begin{proposition}
Let $V=V_1\oplus\ldots\oplus V_s$ be a decomposition into irreducible
summands. Then every simple component $G_i$ acts not identically only on
one summand $V_i$ (up to renumbering), and thus $m=s$. Moreover,
every $G_i$ acts on the set of nonzero vectors in $V_i$ transitively.
\end{proposition}

\begin{proof}
By Proposition~\ref{propprop}, the complement of the open $G$-orbit $U$ in $V$ is a union of
coordinate subspaces (in some, possibly nonlinear, coordinate system). Thus
each irreducible component of the complement
is a smooth variety. The linear action of the group $G$ on $V$
commutes with the group $\kk^{\times}$ of scalar operators, and the
open orbit $U$ as well as any component of the complement
$V\setminus U$ is $(G \times \kk^{\times})$-invariant.
But a cone is a smooth variety if and only if it is a subspace. This shows that
each component of $V\setminus U$ is a maximal proper submodule of $V$.
In particular, the number of maximal
proper submodules is finite and thus the $G$-modules
$V_1,\ldots, V_s$ are pairwise non-isomorphic.
The orbit $U$ is the set of vectors $v \in V$ whose
projection on each $V_i$ is nonzero. This implies that the group $G$
acts on the set of nonzero vectors of each submodule $V_i$ transitively.

If several components of $G$ act on some $V_i$ not identically, then $V_i$
is isomorphic to the tensor product of simple modules of these
components. Then the cone of decomposable tensors in $V_i$ is $G$-invariant,
a contradiction.

Suppose that a simple component $G_l$ acts on both
$V_i$ and $V_j$ not identically. Then $G_l$ acts transitively on the set
of pairs $(v_i, v_j)$ with nonzero $v_i$ and $v_j$.
In particular, any such pair is an eigenvector of a Borel
subgroup of $G_l$. Fix a Borel subgroup $B\subset G_l$ and
a highest vector for $B$ in $V_i$ as $v_i$ and a lowest
vector for $B$ in $V_j$ as $v_j$. Since the intersection of two opposite parabolic
subgroups of $G_l$ does not contain a Borel subgroup, we get a contradiction.
\end{proof}

The following lemma is well known. We give a short self-contained proof
suggested by the referee.

\begin{lemma}
Finite-dimensional rational modules of a simple group $G$ such that
$G$ acts on the set of nonzero vectors transitively are
\begin{enumerate}
\item
the tautological $\SL(n)$-module $\kk^n$ and $\Sp(2n)$-module $\kk^{2n}$;
\item
the dual $\SL(n)$-module $(\kk^n)^*$.
\end{enumerate}
\end{lemma}

\begin{proof}
Since $G$ acts on $V\setminus \{0\}$ transitively, $V$ is a simple $G$-module
of highest weight $\lambda$ and $V={\mathfrak g}v_{\lambda}$, where ${\mathfrak g}$
is the tangent algebra of the group $G$ and $v_{\lambda}$ is a highest weight vector.
In particular, a lowest weight vector is $v_{-\lambda^*}=e_{-\alpha}v_{\lambda}$,
where $\alpha$ is a positive root, whence $\alpha=\lambda+\lambda^*$ is the highest root.
This occurs only for $G=\SL(n)$ with fundamental weights $\lambda=\omega_1, \omega_{n-1}$, and
$G=\Sp(2n)$ with $\lambda=\omega_1$.
\end{proof}

Applying an outer automorphism of $G$, we may assume that
$G=G_1\times\ldots\times G_m$ and $V=V_1\oplus\ldots\oplus V_m$,
where every component $G_i$ is either $\SL(n_i)$ or
$\Sp(n_i)$, and $V_i$ is the
tautological $G_i$-module with identical action of other components.
The open $G$-orbit $U$ in $V$
coincides with the subvariety $\XX = \XX(n_1,\ldots,n_m)$. Therefore
the variety $X$ is obtained from $\XX$ by central factorization.

Let $\mathbb{S} = (\kk^{\times})^m$ be an
algebraic torus acting on $V = V_1\oplus\ldots\oplus V_m$
by component-wise scalar multiplication.
It remains to explain why for any subgroup $S\subseteq \mathbb{S}$
there exists a geometric quotient $\XX \to \XX / S$.
This follows from the fact that $\XX$ is a homogeneous space of the group
$\overline{G} : = \mathrm{GL}(n_1)\times\ldots\times\mathrm{GL}(n_m)$, and $S$ is
a central subgroup of $\overline{G}$.
The proof of Theorem~\ref{main} is completed.

\begin{remark}
The collection $(n_1,\ldots,n_m)$ is determined by a homogeneous
toric variety $X$ uniquely. Indeed, if $\kk^d \supset U \to X$ is the Cox realization
of $X$ and $C_1,\ldots,C_m$ are irreducible components of the complement
$\kk^d \setminus U$, then $n_i = d - \dim C_i$.
\end{remark}

\section{Properties of homogeneous toric varieties}
\label{four}

In this section we use standard notation of toric geometry,
see \cite{Ful}. Let $\NN$ be the lattice of
one-parameter subgroups of a $d$-dimensional torus $\mathbb{T}$
and $\MM$ be the lattice of characters of $\mathbb{T}$.
The torus $\mathbb{T}$ acts diagonally on the space
$\kk^d = V = V_1\oplus\ldots\oplus V_m$, and $\mathbb{S}\subset \mathbb{T}$
is the $m$-dimensional subtorus acting on every $V_i$ by scalar multiplication.
Identification of $\mathbb{T}$
with $(\kk^{\times})^d$ defines standard bases in $\NN$ and $\MM$.
Moreover, the decomposition $V=V_1\oplus\ldots\oplus V_m$
divides the standard basis of $\NN$ into $m$ groups $I_1,\ldots,I_m$,
where each group $I_j$ contains $n_j$ basis vectors and $n_j := \dim V_j$.
The open subvariety $\XX(n_1,\ldots,n_m) = U \subset V$
is a toric $\mathbb{T}$-variety.
Its fan $\mathcal{C}=\mathcal{C}(n_1,\ldots,n_m)$ in the lattice $\NN$
consists of the cones generated by all collections of
standard basis vectors that do not contain any subset $I_j$.

Let $S\subseteq\mathbb{S}$ be a closed subgroup. There is
a sequence of lattices of one-parameter subgroups
$\NN_S\subseteq \NN_{\mathbb{S}}\subset \NN$, where the lattice $\NN_S$ is determined by
the connected component $S^0$ of the quasitorus $S$. The fan
$\mathcal{C}_{S^0}$ of the quotient space $\XX/S^0$ is the image of
the fan $\mathcal{C}$ under the projection
$$
\NN_{\mathbb{Q}} \ \to \ (\NN/\NN_S)_\mathbb{Q}.
$$
The fan $\mathcal{C}_S$ of the variety
$\XX/S$ coincides with the fan $\mathcal{C}_{S^0}$ considered with
regard to an overlattice of $\NN/\NN_S$ of finite index, see
\cite[Section~2.2]{Ful}. In particular, the fan
$\mathcal{C}_\mathbb{S}$ coincides with the fan
$\mathcal{P}$ of the product of
projective spaces
$\mathbb{P}^{n_1-1}\times\ldots\times\mathbb{P}^{n_m-1}$, and
$\mathcal{C}_S$ may be considered as an intermediate step of the projection:
$$
\mathcal{C}\rightarrow\mathcal{C}_S\rightarrow \mathcal{P}.
$$
Let us define a sublattice $\MM_S \subseteq \MM$ as the set of
characters of the torus $\mathbb{T}$ containing $S$ in the kernel.
Elements of $\MM_S$ are linear functions on the
space $(\NN/\NN_S)_\mathbb{Q}$.

\begin{proposition}
Let $X = \XX/S$ be a homogeneous toric variety. Then
\begin{enumerate}
\item{} the variety $X$ is quasiprojective;
\item{} the variety $X$ is not affine;
\item{} the variety $X$ is projective if and only if it coincides with
$\mathbb{P}^{n_1-1}\times\ldots\times\mathbb{P}^{n_m-1}$;
\item{} the variety $X$ is quasiaffine if and only if the lattice $\MM_S$ contains
a vector with positive coordinates;
\item{} the variety $X$ has a nonconstant regular function if and only if
the lattice $\MM_S$ contains a nonzero vector with nonnegative coordinates.
\end{enumerate}
\end{proposition}

\begin{proof}
(1) By Chevalley's Theorem, any homogeneous space of an affine
algebraic group is a quasiprojective variety.

(2) A toric variety obtained via Cox construction is affine if
and only if $U=V$. In our situation this is not the case.

(3) Maximal dimension of a cone in the fan $\mathcal{C}$
equals $n_1+\ldots+n_m-m$. Therefore the fan $\mathcal{C}_S$
is complete if and only if it is obtained from $\mathcal{C}$
by projection to $(\NN/\NN_\mathbb{S})_\mathbb{Q}$, and thus
$\mathcal{C}_S$ coincides with $\mathcal{P}$.

(4) A toric variety is quasiaffine if and only if its fan is
a collection of faces of a strongly convex polyhedral cone.
In our case, this condition implies that the projection $K$ of
the support of the fan $\mathcal{C}$ to $(\NN/\NN_S)_\mathbb{Q}$
is a strongly convex cone. The latter
is equivalent to existence of a linear function on the
space $(\NN/\NN_S)_\mathbb{Q}$ that is positive on $K \setminus \{0\}$.
This gives the desired element of the lattice $\MM_S$.

Conversely, assume that the lattice $\MM_S$ contains
a vector $v$ with positive coordinates.
We have to show that the projection of
each cone of the fan $\mathcal{C}$ is a face of $K$.
Fix proper subsets $J_1\subset I_1,\ldots,J_m \subset I_m$
of the sets of standard basis vectors of the lattice $\NN$.
We claim that
there is an element of the lattice $\MM_S$, which vanishes on
the vectors of $J_1\cup\ldots\cup J_m$ and is positive on other
standard basis vectors. Indeed, the sublattice $\MM_S$ is defined in terms of the sums
of coordinates of a character over all $m$ groups of its coordinates.
The desired vector should have the same sums of coordinates
over the groups as the vector $v$.

(5) Since regular functions on $X$ form a rational
$\mathbb{T}$-module, one may consider only
$\mathbb{T}$-semiinvariant regular functions.
Further, regular $\mathbb{T}$-semiinvariants on $X$ correspond to
characters from $\MM_S$ that are nonnegative on the rays of
the fan $\mathcal{C}$, see \cite[Section~3.3]{Ful}.
\end{proof}

\begin{remark}
Let $X$ be a homogeneous toric variety. Then $X$ is projective if and only
if $X$ contains a $\mathbb{T}$-fixed point. Indeed, the latter condition means
that the fan $\mathcal{C}_S$ contains a cone of full dimension, thus $\NN_S=\NN_{\mathbb{S}}$
and $S=\mathbb{S}$.
\end{remark}

\begin{example}
Let $m=2$ and $n_1=n_2=2$. Then
$\XX = (\kk^2\setminus\{0\})\times(\kk^2\setminus\{0\})$. Set
$S=\{(s,s,s,s):s\in\kk^\times\}.$ Then
$$
\MM_S \, = \, \{(x_1,x_2,x_3,x_4) \, ; \, x_i \in \ZZ, \, x_1+x_2+x_3+x_4 \, = \, 0\},
$$
and the variety $X$ is $\mathbb{P}^3\setminus (D_1\cup D_2)$, where
$D_i \cong \mathbb{P}^1.$ If we set
$S=\{(s,s,s^{-1},s^{-1}):s\in\kk^\times\},$ then
$$
\MM_S \, = \, \{(x_1,x_2,x_3,x_4) \, ; \, x_i \in \ZZ, \, x_1+x_2 \, = \, x_3+x_4\},
$$
and $X$ is a three-dimensional quadratic cone with the apex removed.
\end{example}

Let us characterize the fans of homogeneous toric varieties. Let $N$ be a
lattice, $\Delta$ be a fan in
%the space
$N_{\mathbb{Q}}$ and $P$
be the set of primitive vectors on the rays of
%the fan
$\Delta$. Denote by $N_0$ a sublattice of $N$ generated
by $P$. Fix a positive integer $m$.

\begin{definition}
A fan $\Delta$ is called {\em $m$-partite\/} if
\begin{itemize}
\item{} the set $P$ spans the vector space $N_\mathbb{Q}$;
\item{} the set $P$ can be decomposed into $m$ subsets
$P = I_1 \sqcup \ldots \sqcup I_m $, where each $I_j$ contains
at least two elements, and the cones of
$\Delta$ are exactly the cones generated by subsets
$J\subset P$ that do not contain any $I_j$.
\end{itemize}
\end{definition}

Set $I_j=\{e_1^j,\ldots,e_{n_j}^j\}$ and
$q_j=e_1^j+\ldots+e_{n_j}^j$. Let $Q$ be a sublattice of $N$
generated by $q_1,\ldots,q_m$, and
$Q_\mathbb{Q}=Q\otimes_\mathbb{Z}\mathbb{Q}$.

\begin{proposition}
\label{x}
A fan $\Delta$ is the fan of a homogeneous toric variety if and only if
\begin{enumerate}
\item $\Delta$ is $m$-partite for some $m \ge 1$;
\item every linear relation among  elements of $P$ has the
form $\lambda_1q_1+\ldots+\lambda_mq_m=0$ for some rational $\lambda_i$;
\item $N\subset N_0+Q_\mathbb{Q}$.
\end{enumerate}
\end{proposition}

\begin{proof}
A fan is $m$-partite if and only if it is a projection of the fan
$\mathcal{C}(n_1,\ldots,n_m)$ with some $n_i \ge 2$.
Condition 2 means that the kernel of the projection
is of the form $(\NN_{S^0})_{\QQ}$, where
$S\subseteq\mathbb{S}$. Finally, condition 3 means that
$N$ is generated by $P$ and some elements
$$
\frac{r_{1i}}{R_i}q_1+\ldots+\frac{r_{mi}}{R_i}q_m, \quad \text{where} \ \
r_{ji} \in \ZZ_{\ge 0}, \ \ R_i \in \ZZ_{>0}, \ \ r_{ji} \, < \, R_i, \ \
\text{and} \ \ i = 1,\ldots, l.
$$
Equivalently, the corresponding toric variety
is obtained as the quotient of
the variety \linebreak $\XX(n_1,\ldots,n_m)/S^0$ by an action of the
group $\Gamma = \Gamma_1 \times \ldots \times \Gamma_l$, where $\Gamma_i$ is the cyclic
group of $R_i$-th roots of unity and an element $\epsilon \in \Gamma_i$
multiplies the $j$-th factor of $\XX(n_1,\ldots,n_m)$ by $\epsilon^{r_{ji}}$.
\end{proof}

\section{Some generalizations}
\label{five}

Let a connected reductive group $G$ act on a
toric variety $X$ transitively.
One may assume that $G = G^s \times L$, where
$G^s$ is a simply connected semisimple group, $L$ is a central
torus, and the $G$-action on $X$ is locally effective.
It is well known that
any toric variety $X$ is isomorphic to a direct product
$X_0 \times X_1$, where $X_0$ is a non-degenerate toric variety
and $X_1$ is an algebraic torus.

Let us give a construction of a transitive $G$-action on a toric variety $X$.
Take a $G^s$-homogeneous toric variety $X_0$ with a locally effective
and $G^s$-equivariant action of a quasitorus $L'$.
Fix an inclusion $L'\subseteq L$
into an algebraic torus $L$ as a closed subgroup. The group $G = G^s \times L$ acts
on $X_0\times L$, where $G^s$ acts on the first factor and $L$ acts on the second
one by multiplication.  Consider the $G$-equivariant action of $L'$ on $X_0\times L$
given by $(x_0,l) \mapsto (sx_0, s^{-1}l)$ for every $s\in L'$. Then
$$
X(X_0, G^s, L', L) \, := \, (X_0 \times L)/ L'
$$
is a $G$-homogeneous toric variety.

\begin{proposition}
Let $X$ be a toric variety endowed with a transitive and locally effective action of a connected reductive group $G=G^s\times L$.
Then the non-degenerate factor $X_0$ of $X$
is a $G^s$-homogeneous toric variety. Moreover,
if $L'$ is the stabilizer of a $G^s$-orbit on $X$ in the torus $L$,
then $X$ is $G$-equivariantly isomorphic to $X(X_0, G^s, L', L)$.
\end{proposition}

\begin{proof}
Since the $G^s$- and $L$-actions on $X$ commute, all $G^s$-orbits
are of the same dimension. Let $Y$ be one of these orbits.
Any invertible function on $Y$ is constant.
Consider the above decomposition $X = X_0 \times X_1$.
Since points on $X_1$ are separated by invertible functions,
$Y$ is contained in a subvariety $X_0 \times \{x_1\}$, where $x_1 \in X_1$.
Let $L'$ be the stabilizer of the subvariety $Y$ in the
torus $L$. Then the stabilizer $H$ of a point $x\in Y$ is
contained in the subgroup $G^s\times L'$ and the homogeneous
space $G/H$ projects onto $G/(G^s\times L')\cong L/L'$.
Points on $L/L'$ are separated by invertible
functions, hence $X_0 \times \{x_1\}$ is contained in a fiber of
the projection. But the fibers coincide with $G^s$-orbits on $X$. This
implies $Y = X_0\times \{x_1\}$.

Let us identify the variety $X_0$ with the subvariety $Y \subseteq X$.
Consider the morphism
$$
\varphi \colon X_0 \times L \ \to \ X, \quad (x_0,l) \ \mapsto \ lx_0.
$$
Two pairs $(x_0, l)$ and $(\widetilde{x_0}, \widetilde{l})$ are in the
same fiber of $\varphi$ if and only if
$(\widetilde{x_0}, \widetilde{l})=(sx_0, s^{-1}l)$ with $s=\widetilde{l}^{-1}l$.
This shows that $\varphi$ induces a bijective morphism $X(X_0, G^s, L', L) \to X$.
Clearly, this is an isomorphism of $G$-homogeneous spaces.
\end{proof}

If the subgroup $L'$ is connected, then $L \cong L' \times L''$ with some
complementary subtorus $L''$, and
$X \cong X_0\times L''$. But unlike the case
of algebraic transformation groups of maximal rank
\cite[Theorem~2]{Vin}, this situation does not always
occur. Indeed, one may consider a toric variety
$(\kk^2 \setminus \{0 \}) \times \kk^\times$ with a transitive locally
effective action of the group $\SL(2) \times \kk^\times$ given as
$(g,t) \cdot (v,a) \, = \, (g(tv), \, t^2a)$.

\begin{remark}
It would be interesting to generalize \cite[Theorem~3]{Vin}
and to describe toric varieties with transitive
actions of non-reductive affine algebraic groups.
\end{remark}

Besides homogeneous toric varieties, our method allows to describe
toric varieties with a generically transitive action of a semisimple
group $G$. By Lemma~\ref{lemac}, they are quasitorus
quotients of open subsets of generically transitive
$(G \times \mathbb{S})$-modules.
Such modules are known as $(G \times \mathbb{S})$-prehomogeneous
vector spaces.
For an explicit description, one needs a list of
prehomogeneous vector spaces. The classification results here
are known only under some restrictions on the group and on the module.
For example, if $G$ is simple and the number of
irreducible summands of the module does not exceed three, the
classification is given in a series of papers of
M.Sato, T.Kimura, K.Ueda, T.Yoshigiaki and others.

If the complement of an open $G$-orbit on a toric variety $X$
has codimension at least two in $X$, then $X$ comes from a
$G$-prehomogeneous vector space
(Proposition~\ref{a}). When the group $G$ is simple, the list
of $G$-prehomogeneous vector spaces is obtained in \cite[Theorems~7-8]{Vin1},
and the corresponding toric varieties are described in \cite[Proposition~4.7]{ARH}.
In constrast to the homogeneous case, here appear singular
\cite[Example~5.8]{ARH} and non-quasiprojective \cite[Example~5.9]{ARH} varieties.

\end{document}